\documentclass[11pt]{amsart}

\usepackage[all]{xy}
\usepackage{amsfonts}
\usepackage{eufrak}
\usepackage{amssymb}
\usepackage{amsmath}
\usepackage{mathrsfs}
\usepackage{color}
\usepackage[pagebackref,colorlinks]{hyperref}

\providecommand{\U}[1]{\protect\rule{.1in}{.1in}}

\theoremstyle{plain}
\newtheorem {lemma}{Lemma}[]
\newtheorem {theorem}[lemma]{Theorem}

\newtheorem {corollary}[lemma]{Corollary}

\newtheorem {proposition}[lemma]{Proposition}

\theoremstyle{remark}
\newtheorem {remark}[lemma]{Remark} 
\newtheorem {remarks}[lemma]{Remarks} 

\theoremstyle{definition}

\newcommand{\Spec}{\operatorname{Spec}}

\newcommand{\Cok}{\operatorname{Coker}}
\newcommand{\Ker}{\operatorname{Ker}}
\newcommand{\Imm}{\operatorname{Im}}
\newcommand\CK[1][1]{\operatorname{CK}_{#1}}
\newcommand\ZK[1][1]{\operatorname{ZK}_{#1}}
\newcommand\CG[1][1]{\operatorname{CG}_{#1}}
\newcommand\ZG[1][1]{\operatorname{ZG}_{#1}}
\newcommand\CF{\operatorname{CF}}

\newcommand\ZF{\operatorname{ZF}}
\newcommand\ind{\operatorname{ind}}

\newcommand{\End}{\operatorname{End}}

\newcommand{\mA}{\mathcal A}
\newcommand{\mB}{\mathcal B}
\newcommand{\mO}{\mathcal O}
\newcommand{\mF}{\mathcal F}

\newcommand{\mM}{\mathcal M}

\newcommand{\mC}{\mathcal C}
\newcommand{\mD}{\mathcal D}

\begin{document}
\title{$K$-Theory of Azumaya Algebras over Schemes}

\author{Roozbeh Hazrat}\address{
Department of Pure Mathematics\\
Queen's University\\
Belfast BT7 1NN\\
United Kingdom} \email{r.hazrat@qub.ac.uk}

\author{Raymond T. Hoobler}\address{
Department of Mathematics\\
City College of New York and Graduate Center, CUNY\\
137th Street and Convent Avenue\\
 NY 10031, USA}
   \email{rhoobler@ccny.cuny.edu}

\begin{abstract}
Let $X$ be a connected, noetherian scheme and
$\mathcal{A}$ be a sheaf of Azumaya algebras on $X$ which is a locally free
$\mathcal{O}_{X}$-module of rank $a$. We show that the kernel
and cokernel of $K_{i}\left( X\right) \rightarrow K_{i}\left( \mathcal{A}
\right) $ are torsion groups with exponent $a^{m}$ for some $m$ and any $i\geq 0$, when $X$ is regular or $X$ is of dimension $d$ with an ample sheaf (in this case $m\leq d+1$). As a consequence,  $K_{i}(X,\mathbb Z/m)\cong K_{i}(\mA,\mathbb
Z/m)$, for any $m$ relatively prime to $a$.
\end{abstract}

\begin{thanks}
{The first author acknowledges the support of EPSRC first grant
scheme EP/D03695X/1 and Queen's University PR grant.}
\end{thanks}

\maketitle


An Azumaya algebra over a scheme is a sheaf of algebra which is one
(\'etale) extension away from being a full matrix algebra. This
should indicate that the functors arising from linear algebra would
be similar over the Azumaya algebra and its base algebra. In
\cite{webcor} it was shown that this is the case for Hochschild
homology (over affine schemes). The aim of this note is to show that
we have a similar result for $K$-functors. Indeed, we will show that,
when $X$ is of dimension $d$ with an ample sheaf, e.g. affine or quasi-projective, or $X$ is regular, and $\mathcal{A}$ is a sheaf of
Azumaya algebras on $X$ that is a locally free $\mathcal{O}_{X}$-module of rank $a$, then $K_i(X)$ is isomorphic to $K_i(\mA)$ up to bounded $a$-torsion. In order to prove these results, we first show that the kernel
and cokernel of $K_{i}\left( X\right) \rightarrow K_{i}\left( \mathcal{A}
\right) $ are torsion groups with exponent $a^{m}$ for some $m$.
When $\mA$ is free over  $\mathcal{O}_X$, this is a straightforward argument using Morita theory. When $\mA$ is locally free, and $X$ is of finite dimension with an ample sheaf, we will use an extension of Bass' result on $K$-theory of rings to do this. An alternative argument is given when $X$ is regular, where we use a Mayer-Vietoris sequence to piece together local results into a global one. One reason that we focus on kernel and cokernel of $K_{i}\left( X\right) \rightarrow K_{i}\left( \mathcal{A}\right)$ is that even for $i=1$ and $\mathcal{A}$ being a division algebra the cokernel gives a very interesting group with applications to the group structure of the algebra (see Remark~\ref{dkyo}). 

Here $K_{i}\left( \mathcal{A}\right) $ is the Quillen $K$-theory of the
exact category of sheaves of $\mO_X-$ locally free, left $\mathcal{A}$-modules of finite type. If $\mA$ is non-commutative,  we define a
locally free, left $\mathcal{A}$-module to be a left $\mathcal{A}$-module $%
\mathcal{M}$ such that $\mM$ is locally free as an $\mO_X$-module, and
we use the category $\mA\mbox{-mod}$ and the corresponding exact subcategory of locally free left $\mA$-modules to calculate
$G_{i}\left( \mathcal{A}\right)$ and $K_{i}\left( \mathcal{A}\right)$ respectively.  We follow the
notation in \cite{srinivas} throughout this
note.  
$X$ and $Y$ will always denote schemes, $\mathcal{R}$ and
$\mathcal{S}$ commutative $\mathcal{O}_{X}$-algebras, and
$\mathcal{A}$ and $\mathcal{B}$ possibly non-commutative
$\mathcal{O}_{X}$-algebras that are locally free and of
finite type as $\mathcal{O}_{X}$-modules. Tensor products are over $\mathcal{%
O}_{X}$ unless otherwise indicated. If $X$ is a scheme and $\mA$ is
a not necessarily commutative sheaf of $\mathcal{%
O}_{X}$-algebras, we follow standard notation and denote the category of
left $\mA$-modules (resp. right $\mA$-modules, $\mA$-bimodules) by
$\mA\mbox{-mod}$ (resp. $\mbox{mod-}\mA,$ $\mA\mbox{-mod-}\mA$). If $\mF$ is a locally
free $\mO_X$-module, $\left[ \mF \right]$ denotes its class in
$K_0\left( X \right)$.

Let $\mathcal C$ be a
category, and let $Ab$ be the category of abelian groups. If $F:\mathcal C \rightarrow Ab$
is any functor and $f:A \rightarrow B$ is a morphism in $\mC$, define the functors $\ZF\left( B/A\right)$
and $\CF\left( B/A\right)$ by the exact sequence
\begin{equation}\label{4term}
\begin{split}
\xymatrix{ 0 \ar[r]&\ZF\left( B/A \right)\ar[r]&F(A)\ar[r]^{F(f)}& F(B)\ar[r]&\CF\left( B/A\right)\ar[r]&0}
\end{split}
\end{equation}

In case $\mC$ and $\mD$ are $\mA\mbox{-mod}$ and $\mB\mbox{-mod}$ respectively and $F$ is $G_i$,
we shorten this to $\ZG[i]\left(  \mB/\mA\right)$ and $\CG[i]\left(  \mB/\mA\right)$ respectively. We adopt similar notation, $\ZK[i]\left(  \mB/\mA\right)$ and $\CK[i]\left(  \mB/\mA\right)$, when restricting to the corresponding exact subcategory of locally free, left modules.

If $\phi :\mA\rightarrow \mB$ is an
$\mathcal{O}_{X}$-algebra homomorphism that makes $\mB$ a sheaf of flat $\mA$ modules, then $\mB\in \mB\mbox{-mod-}\mA$ and $\phi ^{i }:K_{i}\left( \mathcal{A}\right) \rightarrow K_{i}\left( \mathcal{B}\right) $
denotes the functorial map defined by $\mathcal{B}\otimes_{\mA} -:\mathcal{A}\mbox{-mod}
\rightarrow \mathcal{B}\mbox{-mod}$. Note that $\ZK[i]\left(  \mB/\mA\right)$ and $\CK[i]\left(  \mB/\mA\right)$ become functors on the category whose objects are $\mO_X$ algebra homomorphisms.  We denote the map defined by
restricting the action of $\mathcal{B}$ to an action of
$\mathcal{A}$ by
$res_{\mathcal{B}}^{\mathcal{A}}:\mathcal{B}\mbox{-mod}\rightarrow
\mathcal{A}\mbox{-mod}$. Since this is exact it induces $\phi _{i
}:K_{i}\left( \mathcal{B}\right) \rightarrow K_{i}\left( \mathcal{A}%
\right) .$

\begin{proposition}\label{exact}
Let $\alpha :\mathcal{A}\rightarrow \mathcal{B},$ $\beta :\mathcal{B}%
\rightarrow \mathcal{C}$ be homomorphisms of possibly non-commutative
sheaves of $\mathcal{O}_{X}$ algebras such that $\mathcal{B}$ is locally
free when considered as a left $\mA$ and also a right $\mA$ module and $\mathcal{C}$ is locally free when considered as a left $\mB$ and also as a right $\mB$ module. Then, for any $i,$ there is an exact sequence

\begin{eqnarray*}
0 &\longrightarrow &\ZK[i]\left( \mathcal{B}/\mathcal{A}\right) \longrightarrow
\ZK[i]\left( \mathcal{C}/\mathcal{A}\right) \longrightarrow \ZK[i]\left(
\mathcal{C}/\mathcal{B}\right)  \\
&\longrightarrow &\CK[i]\left( \mathcal{B}/\mathcal{A}\right) \longrightarrow
\CK[i]\left( \mathcal{C}/\mathcal{A}\right) \longrightarrow \CK[i]\left(
\mathcal{C}/\mathcal{B}\right) \longrightarrow 0.
\end{eqnarray*}
\end{proposition}

\begin{proof}
This is a straghtforward diagram chase once the four term sequence (\ref{4term}) is broken into two short exact sequences in the diagrams:
\begin{equation}
\begin{split}
\xymatrix{ 0\ar[r]&\ZK[i](\mB/\mA)\ar[r]^-{i_{\mB/\mA}}\ar[d]&K_i(\mA)\ar[r]\ar@{=}[d]&\Imm\left( i_{\mB/\mA} \right)\ar[r]\ar[d]&0\\
0\ar[r]&\ZK[i](\mC/\mA)\ar[r]^-{i_{\mC/\mA}}&K_i(\mA)\ar[r]&\Imm\left( i_{\mC/\mA} \right)\ar[r]&0}
\end{split}
\end{equation}
and
\begin{equation}
\begin{split}
\xymatrix{ &&0\ar[d]&&\\
&&\ZK[i](\mC/\mB) \ar[d]&&\\
0\ar[r]&\Imm(i_{\mB/\mA})\ar[r]\ar[d]&K_i(\mB)\ar[r]\ar[d]&\CK[i](\mB/\mA)\ar[r]\ar[d]&0\\
0\ar[r]&\Imm(i_{\mC/\mA})\ar[r]&K_i(\mC)\ar[r]\ar[d]&\CK[i](\mC/\mA)\ar[r]&0\\
&&\CK[i](\mC/\mB)\ar[d]&&\\
&&0&&}
\end{split}
\end{equation}
\end{proof}

In order to effectively use this result we need some information about one
of the groups.

\begin{proposition}\label{str}
Let $X$ be a scheme. Let $\mA$ and $\mB$ be possibly non-commutative sheaves of $\mO_X$-algebras that are both locally free $\mO_X$-modules of finite type. Let $i_{\mA} :\mathcal{A}\rightarrow \mA \otimes \mathcal{B}$ be the homomorphism defined by $i_{\mA}(x)=x \otimes 1$. Then $%
\ZK[i]\left( \mA \otimes \mathcal{B}/\mathcal{A}\right) $ is annihilated by $\left[
\mathcal{B}\right] $ $\in K_{0}\left( \mathcal{O}_X \right) .$ If $\mA$ is an Azumaya algebra, then $\CK[i]\left( \mA \otimes \mathcal{B}/\mathcal{B}\right) $ is annihilated by $\left[
\mathcal{A}\right] $ $\in K_{0}\left( \mathcal{O}_X \right) .$
Moreover
if $\mathcal{F}$ is a locally free sheaf of $\mathcal{O}_{X}$-modules,
then
\begin{eqnarray*}
\ZK[i]\left( \mA \otimes \End\left( \mathcal{F}\right) /\mA\right)
&\cong  &\Ker\left( \left[ \mathcal{F}^{\vee }\right] :K_{i}\left(
\mA\right) \rightarrow K_{i}\left( \mA\right)
\right) \text{ and} \\
\CK[i]\left(\mA \otimes \End\left( \mathcal{F}\right) /\mA\right)
&\cong &K_{i}\left( \mA\right) /\left[ \mathcal{F}^{\vee }%
\right] \cdot K_{i}\left( \mA\right) .
\end{eqnarray*}

\end{proposition}

\begin{proof}
Consider the commutative diagram:
\begin{equation}
\begin{split}
\xymatrix{\mA \otimes \mB\mbox{-mod} \ar[dr]^{res_{\mA \otimes \mB}^{\mA}} \\
\mA\mbox{-mod} \ar[u]^{-\otimes \mB} \ar[r]_{\mB\otimes -}&
\mA\mbox{-mod}}
\end{split}
\end{equation}
Here the horizontal map induces
multiplication by $\left[ \mathcal{B}\right]\in K_0\left( X \right) $ on $K_i\left( \mA \right) $
and so $\left[ \mathcal{B}\right] \cdot \ZK[i]\left( %
\mA \otimes \mB/\mathcal{A}\right) =0.$

If $\mA$ is a sheaf of Azumaya algebras, then $\mA \otimes \mA^{op} \cong \End( \mA )$. The Morita Theorems \cite{knus} then show that the functor $\mA \otimes -:\mB\mbox{-mod}\rightarrow \mA^{op}\otimes\mA\otimes\mB\mbox{-mod}$ is an equivalence of categories in the diagram:
\begin{equation}
\begin{split}
\xymatrix{\mA\otimes\mB\mbox{-mod} \ar[r]^-{\mA^{op} \otimes -}&\mA^{op}\otimes \mA\otimes \mB\mbox{-mod} \ar[rr]^-{res^{\mA \otimes \mB}_{\mA^{op}\otimes \mA\otimes \mB}}&&\mA\otimes\mB\mbox{-mod}\\
&\mB\mbox{-mod} \ar[u]^{\mA \otimes -}_{\cong}\ar[urr]_{\mA \otimes -}&
}
\end{split}
\end{equation}
Evaluating the composite of the functors in the row on an $M \in \mA\otimes \mB\mbox{-mod}$, sends $M$ to $\mA^{op}\otimes M$. Consequently this induces multiplication by $[\mA]=[\mA^{op}] \in K_0(X)$ on $K_i(\mA\otimes\mB)$ and this multiplication factors through $K_i(\mB)\rightarrow K_i(\mA\otimes \mB)$ from which the desired result follows easily.

The Morita Theorems show that $\mathcal{F}\otimes - : \mA\mbox{-mod}
\rightarrow \End\left( \mF\right)\otimes\mA\mbox{-mod}$ is an equivalence
of categories that takes locally free left $\mA$ modules to locally
free, left $\End\left( \mathcal{F}\right)\otimes\mA $ modules.
Consequently $K_{i}\left(\End\left( \mF\right)\otimes\mA \right) $ may be
identified with $K_{i}\left( \mA\right) ,$ and the map $\phi ^{i
}:K_{i}\left( \mA \right) \rightarrow K_{i}\left( \End\left( \mF\right) \otimes \mA\right) $ becomes, with this identification,
\begin{equation*}
\left[ \mathcal{F}^{\vee }\right] :K_{i}\left( \mA\right) \rightarrow
K_{i}\left( \mA\right)
\end{equation*}%
from which the assertion follows.
\end{proof}

\begin{corollary}\label{mcorm}
Let $\mA $ be a sheaf of Azumaya algebras on a  scheme $X.$
Then, for all $i \geq 0$,
\begin{equation*}
\left[ \mA \right] \cdot \ZK[i]\left( \mA /X\right) =0=\left[
\mA \right] \cdot \CK[i]\left( \mA /X\right) .
\end{equation*}
In particular, if $\mA$ is free over $\mathcal{O}_X$ of rank $a$
then
 $\ZK[i]\left(  \mA/\mathcal{O}_{X}\right)$ and  $\CK[i]\left(  \mA/\mathcal{O}_{X}\right)$ are
 torsion abelian groups of exponent dividing $a$.
\end{corollary}

\begin{remark}
Let $A$ be an Azumaya algebra free over its center $R$, of rank $n^2$.
By Corollary~\ref{mcorm}, $$K_{i}(A)\otimes\mathbb Z [1/n]\cong K_{i}(R)\otimes
\mathbb Z[1/n].$$ For the case of a division algebra over a field, this was proved in
Green, et al. \cite{ghr} using the Skolem-Noether theorem, a result of Dawkins
and Halperin on direct limits of division algebras along with several facts
from $K$-theory (see also \cite{hmillar}).
\end{remark}

While this gives us some control over the situation for a general
commutative ring or scheme, it still leaves open the question of
what multiplication by $\left[  \mA\right]  :K_{i}\left(  X\right)
\rightarrow K_{i}\left(  X\right) $ looks like.  We would like to
find a bound when $\mA$ is
 not free over $\mathcal{O}_{X}$.
 The following result which provides some information should be compared to a result of
 Bass (\cite{bass1}, Corollary 16.2).

\begin{proposition}\label{unit}
Let $X$ be a noetherian scheme of dimension $d$ with an ample sheaf. If
$\mathcal{F}$ is a locally free sheaf of rank $f,$ then there is
$Z\in K_{0}\left( X\right) $ such that $\left[ \mathcal{F}\right]
\cdot Z=f^{m}$ for some integer $m\leq d+1$. In particular
\begin{equation*}
\Ker \big ( \left[ \mF \right] :K_{i}\left( X \right ) \rightarrow
K_{i}\left( X\right) \big)
\end{equation*}%
is a torsion group of exponent dividing $f^{m}.$
\end{proposition}

\begin{proof}
Recall,  (\cite{fultonlang}, V Corollary~3.10), that $I^{d+1}=0$ if $I:=\Ker\left[ rk:K_{0}\left(
X \right ) \rightarrow \mathbb{Z}\right]$. Since $X$ is connected,
$\mF$ has constant rank $f$ and so $\left( f-\left[ \mF\right]
\right) \in I$. Hence $\left( f-\left[ \mF\right] \right) ^{m}=0$
for some $m\leq d+1$. Expanding this product and moving $f^{m}$ to
the other side gives the equation
\begin{eqnarray*}
f^{m} &=&\left[ \mF\right] \left( m f^{m-1}+\cdots -(-1)^m\left[ \mF
\right]
^{m-1}\right) \in K_{0}\left( X\right) . \\
&&
\end{eqnarray*}%
Let $Z=\left( m f^{m-1}+\cdots -(-1)^m\left[ \mF \right]
^{m-1}\right)$.
\end{proof}
\begin{remarks}\hfill
\begin{enumerate}
\item If $X$ is not noetherian and finite dimensional but has an ample sheaf, then $X$ can be written as a limit of finite dimensional, noetherian schemes with an ample sheaf, $\mF$ will be defined on one of them, and so the result will still hold but without an explicit bound on $m$.
         \item In the affine case, $X=\Spec\left ( R \right )$, we can say a little
more. Bass (\cite{bass1}, Proposition 15.6) has shown that if $x\in
K_{0}\left( R\right) $ and $rk\left( x\right) \geq d,$ then
$x=\left[ Q\right] $ for some projective module $Q$ where
$d=dim\left( Max\left ( R \right ) \right )$ and if $\left[
Q_{1}\right] =\left[ Q_{2}\right] \in K_{0}\left( R\right) $ and
$rk\left[ Q_{1}\right]
>d,$ then $Q_{1}\cong Q_{2}.$ As a consequence if $rk \left ( Z \right )\geq d$, e.g.
$f^{m-1} \geq d$, then there is a projective $R$ module $Q$ such
that $\mF\otimes Q \cong R^{f^{m}}$.
\end{enumerate}
\end{remarks}

This gives the following result.
\begin{proposition}\label{tors}
Let $X$ be a connected, noetherian scheme of dimension $d$ with an
ample sheaf. If $\mA$ and $\mB$ are sheaves of Azumaya algebras on $X$ which are
locally free of rank $a$ and $b$ as $\mO_{X}$-modules, respectively, 
then, for all $i$, $\ZK[i](\mA\otimes \mB /\mA)$ and $\CK[i](\mA \otimes \mB/\mA)$ are torsion groups of exponent dividing $b^m$ for some integer $m \leq d+1$. In particular,   
$\ZK[i]\left(\mA/X\right) $ and $\CK[i]\left(\mA/X\right) $ are
torsion groups of exponent dividing $a^m$ for some integer $m\leq d+1.$
\end{proposition}
\begin{proof}
By Proposition~\ref{str}, the groups $\ZK[i](\mA\otimes \mB /\mA)$ and $\CK[i](\mA \otimes \mB/\mA)$ are annihilated by $[\mB]\in K_0(X)$. But by Proposition~\ref{unit}, there is a $Z\in K_{i}\left( X\right)$ such that $\left[\mB\right]\cdot Z=b^{m}$ with $m\le d+1$. The second statement is now immediate. 
\end{proof}
The following corollary shows that the groups $\ZK[i]$ and $\CK[i]$ respect the ``primary decomposition". 
\begin{corollary}
Let $\mA$ and $\mB$ be sheaves of Azumaya algebras on a scheme $X$ as in Proposition~\ref{tors}. If $a=rank(\mA)$ is
relatively prime to $b=rank(\mB),$ then $\ZK[i]\left( \mA\otimes \mB/X\right)
\cong \ZK[i]\left( \mA/X\right) \oplus \ZK[i]\left( \mB/X\right) $ and $%
\CK[i]\left( \mA\otimes \mB/X\right) \cong \CK[i]\left( \mA/X\right) \oplus
\CK[i]\left( \mB/X\right) $.
\end{corollary}
\begin{proof}
In the six term exact sequence of Proposition~\ref{exact} associated to $O_{X}\rightarrow
\mA\rightarrow \mA\otimes \mB,$ the first and fourth groups are $a^{n}$ torsion
while the third and sixth groups are $b^{n}$ torsion for $n$ sufficiently
large by Propositions~\ref{str} and \ref{unit}. Hence 
$\ZK[i]\left( \mA\otimes \mB/\mA\right) \rightarrow \CK[i]\left(
\mA/X\right) $ is zero and $\ZK[i]\left( \mA/X\right) $ and $\CK[i]\left(
\mA/X\right) $ are the $a$ primary components of  $\ZK[i]\left( \mA\otimes
\mB/X\right) $ and $\CK[i]\left( \mA\otimes \mB/X\right) $ respectively. The
same argument deals with the $b$ primary part by interchanging $A$ and $\mB$.
\end{proof}
There is an alternative argument using the Mayer-Vietoris exact
sequence to piece together the local results from
Corollary~\ref{mcorm} into a global one. Although for this formulation we don't need the existence of an ample sheaf on $X$, we are, however, forced to restrict our attention to regular, noetherian
schemes. We use the following counting lemma.

\begin{lemma}\label{counting}
Let $A_{i}\overset{f_{i}}{\longrightarrow}B_{i}, 1\leq i \leq5, $ be a
family of maps between exact sequences of abelian groups giving rise
to a commutative diagram:
\begin{equation}
\begin{split}
\xymatrix{
&\cdots \ar[r]& A_{1}  \ar[r]\ar[d]^{f_1} &  A_{2} \ar[r] \ar[d]^{f_2}& A_{3} \ar[r] \ar[d]^{f_3}& A_{4} \ar[r] \ar[d]^{f_4}& A_{5} \ar[r] \ar[d]^{f_5}&   \cdots\\
&\cdots \ar[r] & B_{1} \ar[r] &  B_{2} \ar[r] &
B_{3} \ar[r] & B_{4} \ar[r] & B_{5} \ar[r]&
\cdots}
\end{split}
\end{equation}
Suppose that $m_{i},n_{i}$ annihilate $\Ker\left(f_{i}\right)
,\Cok\left(f_{i}\right)$ respectively for $i=1,2,4,5.$ Then
$\Ker\left(  f_{3}\right)  $ and $\Cok\left(f_{3}\right)$ are
annihilated by $m_{2}n_{1}m_{4}$ and $n_{2}m_{5}n_{4}$ respectively.
\end{lemma}

\begin{proof}
This is a straightforward diagram chase. \qedhere
\end{proof}

We then apply this to our setting.
\begin{lemma}\label{mayer}
Let $X$ be a noetherian scheme and $\mA$ a sheaf of Azumaya algebras
on $X$. Suppose there is a covering of $X$ by open sets $U_{t}$ such
that $\mA$ is free of rank $a$ on each $U_{t}.$ Let $V_{k}=$
$\cup_{t=1}^{k}U_{t}.$ Then $\ZG[i]\left(
\mA\mid_{V_{k}}/V_{k}\right) $ and $\CG[i]\left(
\mA\mid_{V_{k}}/V_{k}\right)  $ are torsion abelian groups of
exponent dividing $a^{k(k-1)+1}$ and $ a ^{2k-1}$ respectively. If
$X$ is a
regular, noetherian scheme, the same bounds apply to $\ZK[i]\left(  \mA\mid_{V_{k}}%
/V_{k})\right)  $ and $\CK[i]\left(  \mA\mid_{V_{k}%
}/V_{k}\right)  .$
\end{lemma}

\begin{proof}
We use induction, the Mayer-Vietoris sequence for $G$-theory
(\cite{srinivas}, 5.16), and Lemma~\ref{counting} to establish this
result for $G$-theory. The argument leading up to
Corollary~\ref{mcorm} works equally well for $G$-theory and shows
that $\ZG[{i}]\left( \mA|_{U_{t}}/U_{t})\right)  $ and $\CG[{i
}]\left( \mA|_{U_{t}}/U_{t}\right)  $ are torsion abelian groups
whose exponent divides $a.$ Then we apply Lemma~\ref{counting} to
the map between
Mayer-Vietoris sequences for $G_{i}\left(  -\right)  $ and $G_{i}%
\left(\mA|_{ - }\right)$ below. We are in the following setting
\begin{equation*}
\begin{split}
\xymatrix{ G_{i+1}\left(  V_{k}\right)  \oplus G_{i+1}\left(
U\right) \ar[r] \ar[d]&  G_{i+1}\left(  V_{k}\cap U\right)
\ar[r]\ar[d]& G_{i}\left(V_{k+1} \right) \ar[r]\ar[d]& G_{i}\left(
V_{k}\right) \oplus G_{i}\left(  U\right)\ar[r]\ar[d]&
G_{i}\left(  V_{k}\cap U\right)\ar[d]\\
G_{i+1}\left(\mA|_{V_{k}}\right)  \oplus
G_{i+1}\left(\mA|_{U}\right) \ar[r]& G_{i+1}\left(\mA|_{V_{k}\cap
U}\right) \ar[r]& G_{i}\left(\mA|_{V_{k+1}} \right) \ar[r]&
G_{i}\left(\mA|_{V_{k}}\right)  \oplus
G_{i}\left(\mA|_{U}\right)\ar[r]& G_{i}\left(\mA|_{V_{k}\cap
U}\right)}
\end{split}
\end{equation*}
where $U=U_{k+1}$. Multiplication by $a$ annihilates the kernel and
cokernel of the second
and fifth columns since $\mA\mid_{V_{k}\cap U_{k+1}%
}\cong {\displaystyle\bigoplus\limits_{1}^{a}}
\mathcal{O}_{V_{k}\cap U_{k+1}}.$ Define $z\left( k\right) $ and
$c\left( k\right) $ to be the integer exponents resulting from
applying Lemma \ref{counting} to this diagram so that
\begin{equation*}
a^{z\left( k\right) }\ZG[i]\left(
\mA\mid_{V_{k}}/\mathcal{O}_{V_{k}}\right)  =0=a^{c\left( k\right)
}\CG[i]\left( \mA\mid_{V_{k}}/\mathcal{O}_{V_{k}}\right)  .
\end{equation*}%
The resulting recursion relations become
\begin{equation}
z\left( k+1\right)=z\left( k\right)+c\left( k\right)+1
\end{equation}
and
\begin{equation}
c\left( k+1\right)=c\left( k\right)+2
\end{equation}
subject to the initial conditions $z\left( 1\right)=1$ and $c\left(
1\right)=1$. Thus $c\left( k\right)=2k-1$ and substituting into the
equation for $z\left( k+1\right)$ yields $z\left( k+1\right)=z\left(
k\right)+2k$. Consequently $z\left( k\right)=k(k-1)+1$.

For the second statement, when $X$ is a regular, noetherian scheme,
$G$-theory coincides with $K$-theory.
\end{proof}
\begin{corollary} \label{realmain}
Let $X$ be a noetherian scheme of dimension $d$ and $\mA$ a sheaf of
Azumaya algebras on $X$ such that $\mA$ is locally free of rank
$a$.Then $\ZG[{i}]\left( \mA/\mathcal{O}_{X}\right)$ and
$\CG[i]\left( \mA/\mathcal{O}_{X}\right)$ are torsion abelian groups
of exponent dividing $  a^{d(d+1)+1}$ and $ a^{2d+1}$ respectively.
If $X$ is also regular, then the same bounds apply to $\ZK[i]\left(
\mA/\mathcal{O}_{X}\right)$ and $\CK[i]\left(
\mA/\mathcal{O}_{X}\right)  .$
\end{corollary}
\begin{proof}
We need only to produce a covering of $X$ by $d+1$ open sets satisfying
the hypotheses of the Lemma~\ref{mayer}. We proceed by induction. If $d=0$, then
$X_{red}$ is a finite set of points and the result is obvious. Let
$\eta_{1},\ldots,\eta_{r}$ be generic points of each irreducible
component of $X.$ Then there is an open set $U_{1}$
containing $\eta_{1},\ldots,\eta_{r}$ such that $\mA\mid_{U_{1}}\cong%
{\displaystyle\bigoplus\limits_{1}^{a}}
\mathcal{O}_{U_{1}}$ as an $\mathcal{O}_{U_{1}}$-module. If
$U_{1}=X,$ we are done. Otherwise $\dim\left(  X-U_{1}\right)  \leq
d-1$, and so $X-U_{1}$ can be covered by $d$ open sets, $\left\{
U_{t}\right\} $ with $2\le t \le d+1$, satisfying the hypothesis of
the lemma. Adding $U_{1}$ to this collection finishes the induction
step.
\end{proof}

Recall the basic property of $K$-theory with coefficients, that for
any scheme $X$, an integer $m$ and $i\geq 1$,  there is a functorial
exact sequence
\begin{equation}
0\longrightarrow K_{i}\left(  X\right)  /m\longrightarrow
K_{i}\left( X,\mathbb{Z} /m\right)
\longrightarrow\,_{m}K_{i-1}\left( X\right)
\longrightarrow0\label{k/m}
\end{equation}

\begin{theorem}\label{main}
Let $X$ be a connected, noetherian scheme and $\mA$  a sheaf of
Azumaya algebras on $X$ such that $\mA$ is locally free of rank $a$. If $X$ has an ample sheaf or $X$ is regular, then $K_{i}(X,\mathbb Z/m)=K_{i}(\mA,\mathbb Z/m)$, for any $m$
relatively prime to $a$ and $i\geq 0$.
\end{theorem}
\begin{proof}
If $\mA$ is free over $X$ of rank $a$, it follows directly from
Corollary~\ref{mcorm} that $\ZK[i]\left( \mA/\mathcal{O}_{X}\right)$
and $\CK[i]\left( \mA/\mathcal{O}_{X}\right)$ are torsion groups of
bounded exponent $a$ (here, there is no need for any assumption on $X$). For a locally free Azumaya sheaf $\mA$, if $X$ has a finite dimension with an ample sheaf then by Proposition~\ref{tors}, $\ZK[i]\left( \mA/\mathcal{O}_{X}\right)$ and $\CK[i]\left(
\mA/\mathcal{O}_{X}\right)$ are torsion groups with exponent
dividing some power of $a$. If $X$ is regular, then
since $X$ is quasi-compact, it can
be covered by a finite number of open subsets $U_t$, such that $\mA$
is free of rank $a$ over $U_t$. Thus by Lemma~\ref{mayer},
$\ZK[i]\left( \mA/\mathcal{O}_{X}\right)$ and $\CK[i]\left(
\mA/\mathcal{O}_{X}\right)$ are torsion groups with exponent
dividing some power of $a$.

Now tensor the exact sequence
$$0 \rightarrow \ZK[i](\mA) \rightarrow K_i(X) \rightarrow K_i(\mA) \rightarrow \CK[i](\mA) \rightarrow 0,$$
with $\mathbb Z[1/a]$. Since $\CK[i](\mA)\otimes \mathbb Z[1/a]$ and
$\ZK[i](\mA)\otimes \mathbb Z[1/a]$ vanish, it follows that
\begin{equation}\label{green}
K_i(X)\otimes \Bbb Z[1/a] \cong K_i(\mA)\otimes \Bbb Z[1/a].
\end{equation}

Now from the exact sequence~(\ref{k/m}) we have
\begin{equation}
\begin{split}
\xymatrix{ 0 \ar[r] & K_i(X) \otimes \Bbb Z/m  \ar[r]\ar[d]
&K_i(X,\Bbb Z/m)
\ar[r]\ar[d] &{}_mK_{i-1}(X) \ar[r]\ar[d]& 0\\
0 \ar[r] & K_i(\mA) \otimes \Bbb Z/m \ar[r] &K_i(\mA,\Bbb Z/m)
\ar[r] &{}_mK_{i-1}(\mA) \ar[r]& 0.}
\end{split}
\end{equation}
Since $a$ and $m$ are relatively prime, one can show that the outer
vertical maps are isomorphisms, and therefore so is the middle
vertical map. For example, for the left vertical map, consider the
diagram below and use the snake lemma  for the two short exact
sequences. It follows that $K_i(X)\otimes \Bbb Z/m\cong
\Imm(\phi_i)\otimes \Bbb Z/m\cong K_i(\mA)\otimes \Bbb Z/m$.
\begin{equation}
\begin{split}
\xymatrix{
0 \ar[r]& \Ker(\phi_i) \ar[r] & K_i(X) \ar[dd]_{\eta_m} \ar[rr]^{\phi_i} \ar@{-->>}[dr]& &K_i(\mA) \ar[dd]^{\eta_m}\ar[r]& \Cok(\phi_i) \ar[r] & 0  \\
          &                             &                     & \Imm(\phi_i)  \ar@{^{(}-->}[ur] \ar@{-->}[dd]^<(.2){\cong} &                             &\\
0 \ar[r]& \Ker(\phi_i) \ar[r] & K_i(X) \ar[rr] \ar@{-->>}[rd]& &K_i(\mA) \ar[r]& \Cok(\phi_i) \ar[r] & 0.  \\
          &                             &                     & \Imm(\phi_i)  \ar@{^{(}-->}[ur]  &                             &\\
           }
\end{split}
\end{equation}
Hence $K_i(X,\Bbb Z/m)\cong K_i(\mA,\Bbb Z/m)$.
\end{proof}

\begin{corollary}
Let $(R,\mathrm m)$ be a Henselian local ring, and let $A$ be an Azumaya algebra over $R$ of rank $n^2$.
Then   $$K_{i}\left(  A ,\mathbb{Z}/d\right)
\cong K_{i}\left(  \overline A,\mathbb{Z}/d\right) $$ where $\overline A = A \otimes R/\mathrm m$,
$d$ is invertible in $R$ and $\left(  n,d\right)  =1$.
\end{corollary}
\begin{proof}
Consider the following commutative diagram where the horizontal arrows are isomorphisms by
Theorem~\ref{main} and the left vertical arrow by Gabber's theorem (see Theorem 1 in \cite{gabber}).
 \begin{equation}
\begin{split}
\xymatrix{
K_i(R,\Bbb Z /d) \ar[r]^{\cong}\ar[d]_{\cong}&  K_i(A,\Bbb Z /d)\ar[d]\\
K_i(R/\mathrm m,\Bbb Z /d) \ar[r]^{\cong}& K_i(\overline A,\Bbb Z /d)}
\end{split}
\end{equation}
The result now follows.
\end{proof}

Here is an amusing application of our results.

\begin{theorem}
Let $\mathcal{O}_{v}$ be a equicharacteristic dvr with residue field
$k$ and field of quotients $K.$ Suppose that $A_{v}$ is an Azumaya
$\mathcal{O}_{v}$
algebra with $A_{v}\otimes _{\mathcal{O}_{v}}K:=A$ and with residue algebra $%
A_{v}\otimes _{\mathcal{O}_{v}}k=A_{k}.$ Then there is a six term exact sequence%

\begin{eqnarray*}
0 &\longrightarrow &\ZK[i]\left(  A_{v}/\mathcal{O}_{v}\right) \longrightarrow
\ZK[i]\left( A/K\right) \longrightarrow \ZK[i-1]\left(  A_{k}/k  \right)  \\
&\longrightarrow &\CK[i]\left(A_{v}/\mathcal{O}_{v} \right) \longrightarrow
\CK[i]\left( A/K\right) \longrightarrow \CK[i-1]\left(A_{k}/k \right) \longrightarrow 0.
\end{eqnarray*}
for any \thinspace $i>0$.
\end{theorem}

\begin{proof}
Panin \cite{panin}  has shown that the Quillen-Gersten sequence is
exact for an equicharacteristic regular local ring or for an Azumaya
algebra over such a ring. Thus the two middle rows in the diagram
below are exact since a dvr has only one prime of height one and the
theorem follows from the snake
lemma.

\begin{equation}
\begin{split}
\xymatrix{ &0&0&0&\\
&\CK[{i}]\left(A_{v}/\mathcal{O}_{v}\right) \ar[u] \ar[r]&\CK[{i}]\left(A/K\right) \ar[u] \ar[r]&\CK[{i-1}]\left(A_{k}/k\right) \ar[u] &\\
0 \ar[r] &K_{i}\left( A_{v}\right) \ar[u] \ar[r] &K_{i}\left( A\right) \ar[u] \ar[r]&K_{i-1}\left(A_{k}\right) \ar[u] \ar[r]&0\\
0 \ar[r] &K_{i}\left( \mathcal{O}_{v}\right) \ar[u] \ar[r] &K_{i}\left( K\right) \ar[u] \ar[r]&K_{i-1}\left(k\right) \ar[u] \ar[r]&0\\
&\ZK[{i}]\left(A_{v}/\mathcal{O}_{v}\right) \ar[u] \ar[r]&\ZK[{i}]\left(A/K \right) \ar[u] \ar[r]&\ZK[{i-1}]\left(A_{k}/k\right) \ar[u] &\\
 &0 \ar[u] &0 \ar[u] &0 \ar[u] &\\
}
\end{split}
\end{equation}
\end{proof}

\begin{corollary}
Let $\mathcal{O}_{v}$ be an equicharacteristic dvr with residue
field $k$ and field of quotients $K.$ Let $A_v$ be an Azumaya algebra
of rank $n^{2}$ over $\mathcal{O}_{v}.$ Then
\begin{equation*}
\#\left( \CK[1]\left( A/K\right) \right) =\#\left( \CK[1]\left( A_{v}/\mathcal{%
O}_{v}\right) \right) \left( \frac{n}{\ind\left( A_{k}\right)
}\right) .
\end{equation*}%
In particular if $A$ is  a division algebra over $K$ and
$\CK[1]\left( A\right) =0,$ then $A_{k}$ is a division algebra.
\end{corollary}

\begin{proof}
If $\ A$ is an Azumaya algebra of rank $n^2$ over a field $L$ such
that $A\cong
M_{r}\left( D\right) $ for some division algebra $D$ over $L,$ then $%
D\otimes _{L}-:K_{0}\left( L\right) \rightarrow K_{0}\left( D\right)
$ is an isomorphism. Consequently $A\otimes _{L}-:K_{0}\left(
L\right) \rightarrow
K_{0}\left( A\right) $ is an injection with cokernel of cardinality
$$r=\frac{%
n}{\ind\left( D\right) }$$
 by Proposition~\ref{str}.
\end{proof}

\begin{remark}\label{dkyo}
For a division algebra $D$ over its center $F$ of index $n$, thanks
to the Dieudonn\'{e} determinant, $\CK(D)=D^{\ast}/K^{\ast
}D^{\prime}$ where $D^{\ast}$ is the multiplicative group of $D$ and
$D^{\prime}$ is its commutator subgroup. This group has a direct
application in solving the open problem whether a multiplicative
group of a division algebra has a maximal subgroup. Indeed, since
$\CK(D)$ is torsion of bounded exponent (by Corollary~\ref{mcorm}),
if it is not trivial, it has maximal subgroups and therefore
$D^{\ast}$ has (normal) maximal subgroups. Thus finding the maximal
subgroups in $D^{\ast}$ reduces to the case that $\CK(D)$ is
trivial. It is a conjecture that $\CK(D)$ is trivial if and only if
$D$ is an ordinary quaternion division algebra over a Pythagorean
field. It was shown that such division algebras do have (non-normal)
maximal subgroups. Thus if the above conjecture is settled
positively, one concludes that the multiplicative group of a
division algebra does have a maximal subgroup (see \cite{hazwad} and
references there).
\end{remark}


\begin{thebibliography}{10000}

\bibitem[B]{bass1} H. Bass, \emph{K-Theory and stable algebras}, Publ. IHES, {\bf 22} (1964), 5--60.



\bibitem[FL]{fultonlang} W. Fulton, S. Lang, Riemann-Roch algebra, Springer-Verlag, New York, 1985.

\bibitem[CW]{webcor}G. Corti\~{n}as, C. Weibel, \emph{Homology of Azumaya
algebras}, Proc. AMS, \textbf{121}, no. 1, (1994), 53--55.

\bibitem[HM]{hmillar} R. Hazrat, J. Millar, \emph{A note on $K$-Theory of
Azumaya algebras}, Comm. Algebra, {\bf 38} (2010), no. 3, 919-926.



\bibitem[HW]{hazwad}R. Hazrat, A. R. Wadsworth, \emph{Nontriviality of certain quotients of $K_1$ groups of division algebras},  J. Algebra, \textbf{312} (2007) 354--361.

\bibitem[G]{gabber}O. Gabber, \emph{$K$-theory of Henselian local rings and
Henselian pairs}, Contemporary Math. vol 126, 1992, 59--70.

\bibitem[GHR]{ghr} S. Green, D. Handelman, P. Roberts,
\emph{$K$-theory of finite dimensional division algebras,} J. Pure
Appl. Algebra, {\bf 12} (1978), no. 2, 153--158.

\bibitem[K]{knus}M.-A. Knus, Quadratic and Hermitian forms over rings,
Springer-Verlag, Berlin, 1991.

\bibitem[S]{srinivas} V. Srinivas, Algebraic $K$-theory, Second edition. Progress in Mathematics, 90. BirkhŠuser Boston, 1996.


\bibitem[P]{panin} I. A. Panin, \emph{The equicharacteristic case of the {G}ersten
conjecture}, Proc. Steklov Inst. Math. {\bf 241}  (2003), 154--163.

\end{thebibliography}
\end{document}